\documentclass[12pt, reqno]{article}
\usepackage{amssymb,amsthm,amsmath,cite}
\usepackage[colorlinks=true]{hyperref}
 \hypersetup{urlcolor=blue, citecolor=red}
\usepackage{mathrsfs}
\usepackage{float}
\usepackage{array}
\usepackage{latexsym,bbm,bbding,dsfont,fancyhdr,enumerate}

\numberwithin{equation}{section}

\newtheorem{lemma}{Lemma}[section]
\newtheorem{theorem}{Theorem}

\theoremstyle{definition}

\newtheorem*{notation}{Notation}

\numberwithin{equation}{section}

\begin{document}

\centerline{\Large\bf Goldbach-Linnik type problems involving one prime,}
\centerline{\Large\bf four prime cubes and powers of 2}

\bigskip\bigskip
\centerline{Xue Han and Huafeng Liu}

\centerline{School of Mathematics and Statistics, Shandong Normal University}
\centerline{ Jinan 250358, Shandong, China}
\centerline{Email: {\tt  han\_xue@stu.sdnu.edu.cn; huafengliu@sdnu.edu.cn}}

\bigskip\bigskip

\date{}

{\bf Abstract.} In this paper, we prove that every pair of sufficiently large odd integers can be represented in the form of a pair of one prime, four prime cubes and $48$ powers of $2$.



\section{Introduction}\label{sec.Introd}

In 1950s, Linnik \cite{Linnik51,Linnik53} approximated the even Goldbach conjecture in a different way by proving that every sufficiently large even integer $N$ can be represented in the form of two primes and $k_{1}$ powers of 2, namely
\begin{equation}\label{1.3}
  N=p_{1}+p_{2}+2^{v_{1}}+2^{v_{2}}+\cdots+2^{v_{k_{1}}}.
\end{equation}
It can be easily deduced that \eqref{1.3} with $k_1=0$ is equivalent to the even Goldbach conjecture. In 1998, Liu, Liu and Wang \cite{LLW} first got the explicit value of $k_{1}$ and showed that $k_{1}=54000$ is admissible in \eqref{1.3}. Later, the value of $k_{1}$ was improved by many scholars. Up to now the best result is $k_{1}=8$ established by Pintz and Ruzsa \cite{PR2020}.

In 2001, Liu and Liu \cite{LL2001} proved that every sufficiently large even integer $N$ can be written as a sum of eight prime cubes and $k_{2}$ powers of $2$, namely
\begin{equation}\label{1.4}
N=p_{1}^{3}+p_{2}^{3}+\cdots+p_{8}^{3}+2^{v_{1}}+2^{v_{2}}+\cdots+2^{v_{k_{2}}}.
\end{equation}
In 2010, Liu and L\"{u} \cite{LL2010} proved that $k_{2}=358$ is admissible in \eqref{1.4}. Subsequently, the value of $k_{2}$ was improved by many scholars. So far the best result is $k_{2}=28$ established by the authors \cite{HLDM}.

As a hybrid problem of \eqref{1.3} and \eqref{1.4}, Liu and L\"{u} \cite{LL2011JNT} proved that every  sufficiently large odd integer $N$ can be written as a sum of one prime, four prime cubes and $k_{3}=106$ powers of $2$, namely
\begin{equation}\label{13333}
N=p_{1}+p_{2}^{3}+p_{3}^{3}+p_{4}^{3}+p_{5}^{3}+2^{v_{1}}+2^{v_{2}}+\cdots+2^{v_{k_{3}}}.
\end{equation}
Recently, Ching and Tsang \cite{CT} proved that $k_{3}=15$ is admissible.

In analytic number theory, when we have a method of handling a representation problem such
as \eqref{13333} described above, it is meaningful and interesting to examine to what
extent the method can be applied to the pairs of corresponding forms. For example, Liu and Tsang \cite{liutsang1989} studied the simultaneous representation of pairs of linear equations in three prime variables. In this paper, we consider the simultaneous representation of pairs of sufficiently large odd integers $N_{1}$ and $N_{2}$ satisfying $ N_{1} \asymp N_{2}$ in the form
\begin{equation}\label{main}
\left\{\begin{array}{l}
N_{1}=p_{1}+p_{2}^{3}+p_{3}^{3}+p_{4}^{3}+p_{5}^{3}+2^{v_{1}}+2^{v_{2}}+\cdots +2^{v_{k}},  \\
N_{2}=p_{6}+p_{7}^{3}+p_{8}^{3}+p_{9}^{3}+p_{10}^{3}+2^{v_{1}}+2^{v_{2}}+\cdots +2^{v_{k}},
\end{array}\right.
\end{equation}
where $k$ is a positive integer. We shall give the acceptable value of $k$ such that the equations \eqref{main} are solvable in the following theorem.

\begin{theorem}\label{thm1}
 The equations \eqref{main} with $k=48$ are solvable for every pair of sufficiently large odd integers $N_{1}$ and $N_{2}$ with $N_{1} \asymp N_{2}$.
\end{theorem}

Result of this type was also considered by Chen \cite{chenxin2022} who proved that $k=231$ is acceptable in \eqref{main}. Our result largely improves Chen's result.

To prove Theorem \ref{thm1}, we mainly apply the Hardy-Littlewood method. The key point is to transfer the weighted number of solutions $\mathrm{R}(N_{1},N_{2})$ (see \eqref{MS}) of \eqref{main} to a positive lower bound of $k$ that satisfies $\mathrm{R}(N_{1},N_{2})>0$. On the major arcs, we establish a detailed numerical estimate of the singular series on average by a meticulous calculation (see Lemma \ref{lemSn}). In order to handle the minor arcs, we prove a new estimate for the integrals involving exponential sums (see Lemma \ref{lemS39}). Also the new value of $\lambda$ in the estimate of the Lebesgue measure of the set $\mathscr{E}_{\lambda}$  in Lemma \ref{lemlambda} plays an important role.

\begin{notation}
Throughout this paper, the letter $p$, with or without subscripts, always represents a prime.
We write $e(x)=e^{2\pi ix}$, and $n\sim N$ means $N<n\leq 2N$.
The letter $\epsilon$ denotes an arbitrarily small positive constant, whose value may not be the same at different occurrences.
\end{notation}

\section{Outline of the proof}\label{sec.sketch}

In this section, we give the outline of the proof of Theorem \ref{thm1} by applying the Hardy-Littlewood method. Throughout this paper, we always take the subscript $i = 1, 2$.
Let
\begin{equation}\label{PQL}
  P_{i}=N_{i}^{\frac{1}{9}-2\epsilon}, \quad Q_{i}=N_{i}^{\frac{8}{9}+\epsilon},\quad L=\frac{\log(N_{1}/\log N_{1})}{\log 2}.
\end{equation}
Then we define the major arcs $\mathfrak{M}_{i}$ and the minor arcs $\mathfrak{m}_{i}$ as
\begin{equation}\label{2.2}
\mathfrak{M}_{i}=\bigcup_{1 \leq q_{i} \leq P_{i}}\bigcup_{\substack{1 \leq a_{i} \leq q_{i} \\ \left(a_{i}, q_{i}\right)=1}} \mathfrak{M}_{i}\left(a_{i}, q_{i}\right),\quad \mathfrak{m}_{i}=[0, 1] \backslash \mathfrak{M}_{i},
\end{equation}
where
\begin{equation*}
\mathfrak{M}_{i}\left(a_{i}, q_{i}\right)=\left\{\alpha_{i} \in[0,1]:\left|\alpha_{i}-\frac{a_{i}}{q_{i}}\right| \leq \frac{1}{q_{i} Q_{i}}\right\}
\end{equation*}
 and
\begin{equation*}
      1\leq a_{i}\leq q_{i}\leq Q_{i}, \ \ (a_{i},q_{i})=1.
\end{equation*}
Further, by a similar argument in Kong and Liu \cite{KL} we define
\begin{equation}\label{M}
\mathfrak{M}=\mathfrak{M}_{1} \times \mathfrak{M}_{2}=\left\{\left(\alpha_{1}, \alpha_{2}\right) \in[0,1]^{2}: \alpha_{1} \in \mathfrak{M}_{1}, \alpha_{2} \in \mathfrak{M}_{2}\right\}
\end{equation}
and
\begin{equation}\label{C(M)}
\mathfrak{m}=[0, 1]^{2} \backslash \mathfrak{M}.
\end{equation}
Let
\begin{equation}\label{UVW1}
U_{i}=\frac{1}{2}\left((1-\eta)N_{i}\right)^{\frac{1}{3}}, \quad V_{i}=\frac{1}{2}(\eta N_{i})^{\frac{1}{3}}, \quad W_{i}=U_{i}^{\frac{5}{18}},
\end{equation}
where $\eta$ is a sufficiently small positive constant. We set
\begin{equation}\label{f11}
  f\left(\alpha_{i},N_{i}\right)=\sum_{p\leq N_{i}}(\log p) e\left(p \alpha_{i}\right),
\end{equation}
\begin{equation}\label{2.4}
S_{3}\left(\alpha_{i},U_{i}\right)=\sum_{p\sim U_{i}}(\log p) e\left(p^{3} \alpha_{i}\right),
\end{equation}
\begin{equation}\label{dS3}
S_{3}\left(\alpha_{i},V_{i}\right)=\sum_{p\sim V_{i}}(\log p) e\left(p^{3} \alpha_{i}\right),
\end{equation}
\begin{equation}\label{dS4}
S_{3}\left(\alpha_{i},W_{i}\right)=\sum_{p\sim W_{i}}(\log p) e\left(p^{3} \alpha_{i}\right),
\end{equation}
\begin{equation*}
G\left(\alpha_{i}\right)=\sum_{4 \leq v \leq L} e\left(2^{v} \alpha_{i}\right), \quad
\mathscr{E}_{\lambda}=\left\{(\alpha_{1},\alpha_{2})\in[0,1]^{2}:|G(\alpha_{1}+\alpha_{2})|\geq \lambda L\right\}.
\end{equation*}

Then we let
\begin{equation}\label{MS}
  \mathrm{R}(N_{1},N_{2})=\sum(\log p_{1})(\log p_{2})\cdots(\log p_{10})
\end{equation}
denote the number of solutions of \eqref{main}, which are weighted by $(\log p_{1})(\log p_{2})\cdots(\log p_{10})$,  in $(p_{1},p_{2},\ldots,p_{10},v_{1},v_{2},\ldots,v_{k})$ such that
\begin{equation*}
\begin{aligned}
&p_{1}\leq N_{1},\quad p_{2}\sim U_{1}, \quad p_{3}\sim V_{1},\quad p_{4},p_{5}\sim W_{1}, \\
&p_{6}\leq N_{2},\quad p_{7}\sim U_{2}, \quad p_{8}\sim V_{2},\quad p_{9},p_{10}\sim W_{2}, \\
&4 \leq v_{1}, v_{2}, \ldots, v_{k} \leq L.
\end{aligned}
\end{equation*}
Thus by the definitions of $\mathfrak{M}, \mathfrak{m}$, $\mathscr{E}_{\lambda}$ and the orthogonality, we can write \eqref{MS} as
\begin{equation}\label{R}
\begin{aligned}
 \mathrm{R}&\left(N_{1}, N_{2}\right) \\
=&\left(\iint\limits_{\mathfrak{M}}+\iint\limits_{\mathfrak{m} \cap \mathscr{E}_{\lambda}}+\iint\limits_{\mathfrak{m} \backslash \mathscr{E}_{\lambda}}\right) f\left(\alpha_{1},N_{1}\right)S_{3}\left(\alpha_{1}, U_{1}\right)S_{3}\left(\alpha_{1}, V_{1}\right)S_{3}^{2}\left(\alpha_{1}, W_{1}\right)
f\left(\alpha_{2},N_{2}\right)\\
&\times S_{3}\left(\alpha_{2},U_{2}\right)S_{3}\left(\alpha_{2}, V_{2}\right)S_{3}^{2}\left(\alpha_{2}, W_{2}\right)
G^{k}\left(\alpha_{1}+\alpha_{2}\right) e\left(-\alpha_{1} N_{1}-\alpha_{2} N_{2}\right)\mathrm{d}\alpha_{1}\mathrm{d}\alpha_{2} \\
:=& \mathrm{R}_{1}\left(N_{1}, N_{2}\right)+ \mathrm{R}_{2}\left(N_{1}, N_{2}\right)+ \mathrm{R}_{3}\left(N_{1}, N_{2}\right).
\end{aligned}
\end{equation}

In the following sections, we shall give the desired estimates for these three terms $ \mathrm{R}_{r}\left(N_{1}, N_{2}\right), r=1,2,3$ at the right side of \eqref{R} and thus determine the acceptable value of $k$.

\section{Auxiliary lemmas}
\label{sec.lemmas}

In this section, we give some auxiliary lemmas which will be used in the proof of Theorem \ref{thm1}.
Let $j$ be a positive integer throughout this section. Let also
\begin{equation}\label{An}
C_{j}(q, a)=\sum_{\substack{m=1 \\(m,q)=1}}^{q} e\left(\frac{a m^{j}}{q}\right),\
A(n, q)=\frac{1}{ \varphi^{5}(q)}\sum_{\substack{a=1\\(a, q)=1}}^{q}C_{1}(q, a) C_{3}^{4}(q, a) e\left(-\frac{a n}{q}\right)
\end{equation}
and
\begin{equation}\label{3.1}
\mathfrak{S}(n)=\sum_{q=1}^{\infty} A(n, q).
\end{equation}
By direct calculation, we have $C_{1}(q, a)=\mu(q)$.

\begin{lemma}\label{lem1}
Let $\mathfrak{M}_{i}$, $f(\alpha_{i},N_{i})$, $S_{3}\left(\alpha_{i},U_{i}\right)$, $S_{3}\left(\alpha_{i},V_{i}\right)$ and $S_{3}\left(\alpha_{i},W_{i}\right)$ be defined as in \eqref{2.2}, \eqref{f11}, \eqref{2.4}, \eqref{dS3} and \eqref{dS4}, respectively. Then for $2\leq n_{i}\leq N_{i}$, we have
\begin{equation*}
\int\limits_{\mathfrak{M}_{i}} f(\alpha_{i},N_{i})S_{3}(\alpha_{i},U_{i}) S_{3}(\alpha_{i},V_{i}) S_{3}^{2}(\alpha_{i},W_{i}) e\left(-n_{i} \alpha_{i}\right) \mathrm{d} \alpha_{i}=\frac{1}{3^{4}} \mathfrak{S}(n_{i})\mathfrak{J}(n_{i})+O\left(U_{i}V_{i}W_{i}^{2} L^{-1}\right),
\end{equation*}
where $\mathfrak{S}\left(n_{i}\right)$ is defined as \eqref{3.1} and satisfies $\mathfrak{S}\left(n_{i}\right)\gg 1$ for $n_{i} \equiv 1 \pmod 2$, and $\mathfrak{J}\left(n_{i}\right)$ is defined as
\begin{equation*}
\mathfrak{J}(n_{i}):=\sum_{\substack{m_{1}+m_{2}+\cdots+m_{5}=n_{i}\\U_{i}^{3}<m_{2}\leq 8U_{i}^{3}\\V_{i}^{3}<m_{3}\leq 8V_{i}^{3}\\W_{i}^{3}<m_{4},m_{5}\leq 8W_{i}^{3}}}\left(m_{2}m_{3}m_{4}m_{5}\right)^{-\frac{2}{3}}
\end{equation*}
and satisfies $U_{i}V_{i}W_{i}^{2}\ll \mathfrak{J}\left(n_{i}\right)\ll U_{i}V_{i}W_{i}^{2}$.
\end{lemma}
\begin{proof}
The proof of this lemma is a standard application of the iterative argument developed by Liu and Zhan (see \cite{LJY,LZ}, etc.). Thus we omit its proof here.
\end{proof}

\begin{lemma}\label{lemJn}
For $(1-\eta)N_{i}\leq n_{i}\leq N_{i}$, we have
\begin{equation*}
  \mathfrak{J}(n_{i})\geq 3^{4}(\sqrt[3]{7}-1)(1-\eta)^{4}U_{i}V_{i}W_{i}^{2}.
\end{equation*}
\end{lemma}
\begin{proof}
The domain of the sum $\mathfrak{J}(n_{i})$ can be written as
\begin{equation*}
\mathfrak{D}=\left\{\left(m_{1}, m_{2},\ldots, m_{5}\right): \begin{array}{l}
m_1 \leq N_{i}, U_{i}^{3}<m_{2} \leq 8 U_{i}^{3}, V_{i}^{3}<m_{3} \leq 8 V_{i}^3, \\
W_{i}^3<m_{4},m_{5} \leq 8 W_{i}^3, m_{1}=n_{i}-m_{2}-m_{3}-m_{4}-m_{5}
\end{array}\right\} .
\end{equation*}
Let
\begin{equation*}
\mathfrak{D}^{*}=\left\{\left(m_{1}, m_{2},\ldots, m_{5}\right): \begin{array}{l}
U_{i}^{3}<m_{2} \leq 7 U_{i}^{3}, V_{i}^{3}<m_{3} \leq 8 V_{i}^3, \\
W_{i}^3<m_{4},m_{5} \leq 8 W_{i}^3, m_{1}=n_{i}-m_{2}-m_{3}-m_{4}-m_{5}
\end{array}\right\} .
\end{equation*}
For $(m_{1}, m_{2},\ldots, m_{5})\in \mathfrak{D}^{*}$, we deduce from $(1-\eta)N_{i}\leq n_{i}\leq N_{i}$ that
\begin{equation*}
  m_{1}=n_{i}-m_{2}-m_{3}-m_{4}-m_{5}\leq N_{i}.
\end{equation*}
Thus $\mathfrak{D}^{*}$ is a subset of $\mathfrak{D}$. Then we have
\begin{equation*}
\begin{aligned}
\mathfrak{J}(n_i) & \geq \sum_{(m_{2}, m_{3}, m_{4}, m_{5}) \in \mathfrak{D}^{*}} (m_{2}m_{3}m_{4}m_{5})^{-\frac{2}{3}} \\
& \geq \sum_{U_{i}^{3}<m_{2} \leq 7 U_{i}^{3}} m_{2}^{-\frac{2}{3}} \sum_{V_{i}^{3}<m_{3} \leq 8V_{i}^{3}} m_{3}^{-\frac{2}{3}}\sum_{W_{i}^{3}<m_{4},m_{5}\leq 8U_{i}^{3}} (m_{4}m_{5})^{-\frac{2}{3}}\\
&\geq 3^{4}(\sqrt[3]{7}-1)(1-\eta)^{4}U_{i}V_{i}W_{i}^{2},
\end{aligned}
\end{equation*}
which completes the proof of this lemma.
\end{proof}

\begin{lemma}\label{lemSn}
Let $\Xi(N_{i},k)=\{n_{i}\geq 2: n_{i}=N_{i}-2^{v_{1}}-2^{v_{2}}-\cdots-2^{v_{k}}, 4\leq v_{1},v_{2},\ldots,v_{k}\leq L\}$ with $k\geq 45$. Then for $N_{1} \equiv N_{2} \equiv 1 \pmod 2$, we have
\end{lemma}
\begin{equation}\label{SSn}
\sum_{\substack{n_{1} \in \Xi\left(N_{1}, k\right)\\n_{2} \in \Xi\left(N_{2}, k\right)\\ n_{1} \equiv n_{2} \equiv 1 \pmod 2}} \mathfrak{S}\left(n_{1}\right) \mathfrak{S}\left(n_{2}\right) \geq 3.71280584L^{k}.
\end{equation}

\begin{proof}
From (3.11) in Liu and L\"{u} \cite{LL2011JNT}, we know that $A(n_{i},p)$ is multiplicative and
\begin{equation}\label{S}
  \mathfrak{S}(n_{i})=\prod_{p\geq 2}(1+A(n_{i},p)).
\end{equation}
With the help of computer, we have
\begin{equation*}
\begin{aligned}
  &1+A(n_{i},3)\geq 0.9375,\ 1+A(n_{i},5)\geq 0.99609375,\ 1+A(n_{i},7)\geq 0.72916666, \\
  & 1+A(n_{i},11)\geq 0.9999,\ 1+A(n_{i},13)\geq 0.93098958,\ 1+A(n_{i},17)\geq 0.99998474,\\
  &1+A(n_{i},19)\geq 0.98225308, \ 1+A(n_{i},23)\geq 0.99999573,\ 1+A(n_{i},29)\geq 0.99999837,\\
  &1+A(n_{i},31)\geq  0.99303333,\  1+A(n_{i},37)\geq0.99387538,\  1+A(n_{i},41)\geq 0.9999996,\\
  &\cdots\cdots\\
  &1+A(n_{i},193)\geq 0.99982076,\ 1+A(n_{i},197)\geq 0.99999999, \ 1+A(n_{i},199)\geq 0.99983815.
\end{aligned}
\end{equation*}
Thus, we can find that when $3\leq p\leq 200$ and $ p\neq3,7,13,19$, the value of $1+A(n_{i},p)$ is very close to 1. In addition, we can obtain
\begin{equation}\label{A(5)}
   (1+A(n_{i},5))(1+A(n_{i},11))(1+A(n_{i},17))\prod_{23\leq p\leq 200} 1+A(n_{i},p)\geq 0.96976071.
\end{equation}
From (3.15) and (3.17) in Liu and L\"{u} \cite{LL2011JNT}, we note that
if $p\geq 5$, $p \equiv 2\pmod 3$ and $(a,p)=1$, then
\begin{equation*}
1+A(n_{i}, p)=
\begin{cases}1-\frac{1}{(p-1)^{4}}, & p \mid n_{i}, \\
 1+\frac{1}{(p-1)^{5}}, & p \nmid n_{i},
 \end{cases}
\end{equation*}
and if $p\geq 13$ and $p \equiv 1\pmod 3$, then
\begin{equation*}
1+A(n_{i}, p)>1-\frac{(2 \sqrt{p}+1)^{4}}{(p-1)^{4}}.
\end{equation*}
Then with the help of computer again we get
\begin{equation}\label{200+}
\begin{aligned}
&\prod_{\substack{200 \leq p<10^{6}}}(1+A(n_{i}, p)) \\
&\geq \prod_{\substack{200 \leq p<10^{6} \\
p \equiv 1\pmod 3}}\left(1-\frac{(2 \sqrt{p}+1)^{4}}{(p-1)^{4}}\right) \prod_{\substack{200 \leq p<10^{6} \\
p \equiv 2\pmod 3 \\
p \mid n}}\left(1-\frac{1}{(p-1)^{4}}\right) \prod_{\substack{200 \leq p<10^{6} \\
p \equiv 2\pmod 3 \\
p \nmid n}}\left(1+\frac{1}{(p-1)^{5}}\right) \\
&\geq \prod_{\substack{200 \leq p<10^{6} \\
p \equiv 1\pmod 3}}\left(1-\frac{(2 \sqrt{p}+1)^{4}}{(p-1)^{4}}\right) \prod_{\substack{200 \leq p<10^{6} \\
p \equiv 2\pmod 3}}\left(1-\frac{1}{(p-1)^{4}}\right) \\
&\geq 0.99351588.
\end{aligned}
\end{equation}
From Liu and L\"{u} \cite[p. 726]{LL2011JNT}, we have that if $p\geq 1138$ and $p\equiv 1\pmod 3$, then
\begin{equation*}
1-\frac{(2 \sqrt{p}+1)^{4}}{(p-1)^{4}} \geq\left(1-\frac{1}{(p-1)^{2}}\right)^{17}.
\end{equation*}
Then we get
\begin{equation}\label{100000+}
\begin{aligned}
\prod_{p \geq 10^{6}}(1+A(n_{i}, p)) & \geq \prod_{\substack{p \geq 10^{6}\\ p \equiv 1\pmod 3}}\left(1-\frac{1}{(p-1)^{2}}\right)^{17}\prod_{\substack{p \geq 10^{6}\\ p\equiv 2\pmod 3}}\left(1-\frac{1}{(p-1)^{4}}\right) \\
&>\prod_{m \geq 10^{6}+1}\left(1-\frac{1}{(m-1)^{2}}\right)^{17}\geq 0.999983.
\end{aligned}
\end{equation}
Form \eqref{A(5)}, \eqref{200+} and \eqref{100000+}, we have
\begin{equation}\label{C}
  (1+A(n_{i},5))(1+A(n_{i},11))(1+A(n_{i},17))\prod_{p\geq 23}(1+A(n_{i},p))\geq 0.96345628:=C.
\end{equation}
Let $q=3\times7\times13\times19=5187$, by \eqref{S} and \eqref{C} we get
\begin{equation}\label{SS}
\begin{aligned}
&\sum_{\substack{n_{1} \in \Xi\left(N_{1}, k\right) \\ n_{2} \in \Xi\left(N_{2}, k\right) \\ n_{1} \equiv n_{2} \equiv 1 \pmod 2}}
\mathfrak{S}\left(n_{1}\right) \mathfrak{S}\left(n_{2}\right)\\
&\geq\left(2C\right)^{2} \sum_{\substack{n_{1} \in \Xi\left(N_{1}, k\right) \\ n_{2} \in \Xi\left(N_{2}, k\right) \\ n_{1} \equiv n_{2} \equiv 1 \pmod 2}} \prod_{1\leq i \leq 2} \prod_{ p_{i}=3,7,13,19}\left(1+A\left(n_{i}, p_{i}\right)\right)\\
&\geq \left(2C\right)^{2} \sum_{\substack{1 \leq j \leq q}} \sum_{\substack{n_{1} \in \Xi\left(N_{1}, k\right) \\ n_{2} \in \Xi\left(N_{2}, k\right) \\ n_{1} \equiv n_{2} \equiv 1 \pmod 2 \\ n_{1} \equiv n_{2} \equiv j \pmod q}}\prod_{1\leq i \leq 2}\prod_{ p_{i}=3,7,13,19}\left(1+A\left(j, p_{i}\right)\right)\\
&\geq \left(2C\right)^{2}\sum_{1 \leq j\leq q} \prod_{p=3,7,13,19}(1+A(j, p))^{2}
\sum_{\substack{n_{1} \in \Xi\left(N_{1}, k\right) \\ n_{1}  \equiv 1 \pmod 2\\n_{1} \equiv j \pmod q}} 1.\\
\end{aligned}
\end{equation}
To estimate the innermost sum at the right side of \eqref{SS}, we take a similar argument to Lemma 4.4 in Zhao \cite{ZLL2014}.
We can deduce that
\begin{equation*}
S:=\sum_{\substack{n_{1} \in \Xi\left(N_{1}, k\right) \\ n_{1} \equiv 1 \pmod 2\\n_{1} \equiv j \pmod q}} 1=\left(\frac{L}{\delta(q)}+O(1)\right)^{k}\sum_{\substack{1\leq v_{1},v_{2},\ldots,v_{k}\leq \delta(q)\\2^{v_{1}}+2^{v_{2}}+\cdots+2^{v_{k}}\equiv N-j \pmod q}}1,
\end{equation*}
where $\delta(q)$ denotes the smallest positive integer $\delta$ such that $2^{\delta}\equiv 1 \pmod q$.
Noting that
\begin{equation*}
S = \frac{1}{q}\left(\frac{L}{\delta(q)}+O(1)\right)^{k}  \sum_{t=0}^{q-1} e\left(\frac{t (N-j)}{q}\right)\theta^{k}(t),
\end{equation*}
we get
\begin{equation*}
\begin{aligned}
S & \geq \frac{1}{q}\left(\frac{L}{\delta(q)}+O(1)\right)^{k}\left(\delta(q)^{k}-(q-1)\left(\max_{0<t\leq q-1}\left|\theta(t)\right|\right)^{k}\right) \\
&\geq \frac{L^{k}}{q}\left(1-(q-1)\left(\frac{\max\limits_{0<t\leq q-1}\left|\theta(t)\right|}{\delta(q)}\right)^{k}\right)+O\left(L^{k-1}\right),
\end{aligned}
\end{equation*}
where $\theta(t)=\sum\limits_{1 \leq s \leq \delta(q)} e\left(\frac{t 2^{s}}{q}\right)$.
Recalling the definition of $\delta(q)$, we have
\begin{equation*}
\delta(q)=36 \quad \text{and}\quad  \max_{0<t\leq q-1}|\theta(t)|=18.00001822\ldots.
\end{equation*}
Therefore, we can get
\begin{equation*}
S \geq  0.00019278L^{k}.
\end{equation*}
From \eqref{SS} and
\begin{equation*}
\begin{aligned}
\sum_{1\leq j\leq p}(1+A(j, p))^{2} &=p+2 \sum_{1\leq j\leq p} A(j, p)+\sum_{1\leq j\leq p}(A(j, p))^{2} \\
&=p+\sum_{1\leq j\leq p}(A(j, p))^{2} \\
& \geq p,
\end{aligned}
\end{equation*}
we have
\begin{equation*}
\begin{split}
\sum_{\substack{n_{1} \in \Xi\left(N_{1}, k\right) \\ n_{2} \in \Xi\left(N_{2}, k\right) \\ n_{1} \equiv n_{2} \equiv 1 \pmod 2}} \mathfrak{S}\left(n_{1}\right) \mathfrak{S}\left(n_{2}\right) \geq \left(2C\right)^{2}\cdot 0.00019278\cdot qL^{k} \geq 3.71280584L^{k},
\end{split}
\end{equation*}
which completes the proof of this lemma.
\end{proof}

\begin{lemma}\label{lemS}
Suppose that $\alpha_{i}$ is a real number, and that there exist integers $a_{i}\in\mathbb{Z}$ and $q_{i}\in\mathbb{N}$ with
$$(a_{i},q_{i})=1,\quad 1\leq q_{i}\leq N_{i}^{\frac{1}{2}},\quad |q_{i}\alpha_{i}-a_{i}|\leq N_{i}^{-\frac{1}{2}}.$$
Then we have
\begin{equation}\label{S31}
S_{3}(\alpha_{i},X) \ll X^{1-\frac{1}{12}+\varepsilon}+\frac{X^{1+\varepsilon}}{\sqrt{q_{i}\left(1+N_{i}\left|\alpha_{i}-\frac{a_{i}}{q_{i}}\right|\right)}},
\end{equation}
where $X$ can be taken as $U_{i}$ or $V_{i}$.
\end{lemma}

\begin{proof}
This lemma is Lemma $2.3$ in Zhao \cite{ZLL1}.
\end{proof}

\begin{lemma}\label{lemfS3}
Let $\mathfrak{m}_{i}$ be defined as in \eqref{2.2}. Then we have
\begin{equation}\label{S2}
\max _{\alpha_{i} \in \mathfrak{m}_{i}}|S_{3}(\alpha_{i},U_{i})| \ll N_{i}^{\frac{11}{36}+\epsilon},
\end{equation}
\begin{equation}\label{S3}
\max _{\alpha_{i} \in \mathfrak{m}_{i}}|S_{3}(\alpha_{i},V_{i})| \ll N_{i}^{\frac{11}{36}+\epsilon}.
\end{equation}
\end{lemma}

\begin{proof}
The proof of \eqref{S3} is similar to that of \eqref{S2}, so we only present the proof of \eqref{S2}.
By Dirichlet's lemma on rational approximations, we can deduce that for any real number $\alpha_{i}\in \mathfrak{m}_{i}$, there exist integers $a_{i}$ and $q_{i}$ such that
\begin{equation*}
  \alpha_{i}=\frac{a_{i}}{q_{i}}+\lambda_{i},\ 1\leq q_{i}\leq \mathscr{Q}_{i}=N_{i}^{\frac{1}{2}}, \ |\lambda_{i}|\leq \frac{1}{q_{i}\mathscr{Q}_{i}}, \ (a_{i},q_{i})=1.
\end{equation*}
Since $\alpha_{i}\in \mathfrak{m}_{i}$, based on \eqref{PQL} and \eqref{2.2}, we observe that if $q_{i}\leq P_{i}=N_{i}^{\frac{1}{9}-2\epsilon}$, then $|\lambda_{i}|>\frac{1}{q_{i}Q_{i}}$; otherwise $q_{i}>P_{i}$.
In either case, we have
\begin{equation*}
  q_{i}^{\frac{1}{2}}(1+N_{i}|\lambda_{i}|)^{\frac{1}{2}}>\min(P_{i}^{\frac{1}{2}},N_{i}^{\frac{1}{2}}Q_{i}^{-\frac{1}{2}})=\min(N_{i}^{\frac{1}{18}-\epsilon},
  N_{i}^{\frac{1}{18}+\epsilon}).
\end{equation*}
This in combination with Lemma \ref{lemS} with $X=U_{i}$ leads to
\begin{equation*}
\begin{split}
  \max_{\alpha_{i} \in \mathfrak{m}_{i}}|S_{3}(\alpha_{i},U_{i})|&\ll N_{i}^{\frac{1}{3}-\frac{1}{36}+\epsilon}+N_{i}^{\frac{1}{3}-\frac{1}{18}+\epsilon}
  \ll N_{i}^{\frac{11}{36}+\epsilon}.
\end{split}
\end{equation*}
Thus we complete the proof of \eqref{S2}.
\end{proof}

\begin{lemma}\label{lem4.3}
For $k\geq 3$, let $\mathscr{M}$ be the union of intervals $\mathscr{M}(q,a)$ for
$$1\leq a\leq q\leq P^{k2^{1-k}}, \ \ (a,q)=1,$$
where
$$\mathscr{M}(q,a)=\{\alpha:|q\alpha-a|\leq P^{k(2^{1-k}-1)}\}.$$
For $u\geq 0$, define
\begin{equation*}
  \begin{split}
     \omega_{k}(p^{uk+v})=
     \begin{cases}
       kp^{-u-\frac{1}{2}}, & v=1,\\
       p^{-u-1}, & 2\leq v\leq k,
     \end{cases}
  \end{split}
\end{equation*}
and
$$\mathscr{J}_{0}=\sup_{\beta\in[0,1)}\int_{\mathscr{M}}\frac{\omega_{k}^{2}(q)|h^{2}(\alpha+\beta)|}{(1+P^{k}|\alpha-\frac{a}{q}|)^{2}}\mathrm{d}\alpha.$$
Suppose that $G(\alpha)$ and $h(\alpha)$ are integrable functions of period one. Let
\begin{equation*}
  g(\alpha)=g_{\mathcal{A}}(\alpha)=\sum_{x\in\mathcal{A}}e(x^{k}\alpha), \ \ \mathcal{A}\subseteq(P,2P]\cap\mathbb{N}
\end{equation*}
and $\mathfrak{m}\subseteq[0,1)$ be a measurable set. Then we have
\begin{equation*}
  \int_{\mathfrak{m}}g(\alpha)G(\alpha)h(\alpha)\mathrm{d}\alpha\ll P\mathscr{J}_{0}^{\frac{1}{4}}\left(\int_{\mathfrak{m}}|G(\alpha)|^{2}\mathrm{d}\alpha\right)^{\frac{1}{4}}
  \mathscr{J}^{\frac{1}{2}}(m)+P^{1-2^{-k}+\varepsilon}\mathscr{J}(m),
\end{equation*}
where
$$\mathscr{J}(m)=\int_{\mathfrak{m}}|G(\alpha)h(\alpha)|\mathrm{d}\alpha.$$
\end{lemma}

\begin{proof}
  This lemma is Lemma 3.1 in Zhao \cite{ZLL2}.
\end{proof}

\begin{lemma}\label{lem4.4}
For $\gamma\in\mathbb{R}$, we define
\begin{equation*}
  \mathcal{L}(\gamma)=\sum_{q\leq V_{i}}\sum_{\substack{1\leq a\leq q\\ (a,q)=1}}\int_{|\alpha-\frac{a}{q}|\leq V_{i}}
  \frac{\omega_{3}^{2}(q)d^{c}(q)|\sum_{V_{i}< p\leq 2V_{i}}e(p^{3}(\alpha+\gamma))|^{2}}{1+|\alpha-\frac{a}{q}|V_{i}^{3}}\mathrm{d}\alpha,
\end{equation*}
we have uniformly for $\gamma\in\mathbb{R}$ that
\begin{equation*}
  \mathcal{L}(\gamma)\ll V_{i}^{2}N_{i}^{-1+\epsilon}.
\end{equation*}
Here $c$ is an absolute constant.
\end{lemma}

\begin{proof}
 We can deduce this lemma from Lemma $2.2$ with $k=3$ and $P=Q=V_{i}$ in Zhao \cite{ZLL2}.
\end{proof}

\begin{lemma}\label{lemS39}
We have
\begin{equation*}
  \int_{\mathfrak{m}_{i}}\left|S_{3}^{3}(\alpha_{i},U_{i})S_{3}^{9}(\alpha_{i},V_{i})\right|\mathrm{d}\alpha_{i}\ll N_{i}^{\frac{103}{36}+\epsilon}.
\end{equation*}
\end{lemma}

\begin{proof}
Write
\begin{equation*}
\mathscr{J}(t)=\int_{\mathfrak{m}_{i}}\left|S_{3}^{t}(\alpha_{i},U_{i})S_{3}^{9}(\alpha_{i},V_{i})\right|\mathrm{d}\alpha_{i}, \ \ 1\leq t\leq 3.
\end{equation*}
By Lemma \ref{lem4.3} with
$$g(\alpha)=S_{3}(\alpha_{i},U_{i}),\ \ h(\alpha)=S_{3}(\alpha_{i},V_{i}),\ \ G(\alpha)=|S_{3}^{2}(\alpha_{i},U_{i})S_{3}^{8}(\alpha_{i},V_{i})|,$$
we obtain
\begin{equation}\label{H1H2}
\begin{split}
\mathscr{J}(3)=N_{i}^{\frac{1}{3}}\mathscr{J}_{0}^{\frac{1}{4}}\left(\int_{\mathfrak{m}_{i}}|S_{3}^{4}(\alpha_{i},U_{i})S_{3}^{16}(\alpha_{i},V_{i})|
\mathrm{d}\alpha_{i}\right)^{\frac{1}{4}}\mathscr{J}^{\frac{1}{2}}(2)+N_{i}^{\frac{7}{24}+\epsilon}\mathscr{J}(2).
\end{split}
\end{equation}
Here
\begin{equation*}
  \mathscr{J}_{0}=\sup_{\beta\in[0,1)}\sum_{q\leq V_{i}^{\frac{3}{4}}}\sum_{\substack{1\leq a\leq q\\ (a,q)=1}}
  \int_{\mathscr{M}(q,a)}\frac{\omega_{3}^{2}(q)|h^{2}(\alpha+\beta)|}{(1+V_{i}^{3}|\alpha-\frac{a}{q}|)^{2}}\mathrm{d}\alpha
\end{equation*}
and
$$\mathscr{M}(q,a)=\left\{\alpha:|q\alpha-a|\leq V_{i}^{-\frac{9}{4}}\right\}.$$
We apply Lemma \ref{lem4.4} and get
\begin{equation}\label{c.5}
  \mathscr{J}_{0}\ll \mathcal{L}(\gamma)\ll V_{i}^{2}N_{i}^{-1+\epsilon}\ll N_{i}^{-\frac{1}{3}+\epsilon}.
\end{equation}
For $\mathscr{I}(2)$, by Lemma \ref{lemfS3}, Cauchy's inequality and Hua's inequality, we obtain
\begin{equation}\label{c.6}
\begin{split}
  \mathscr{J}(2)&\leq\mathscr{J}^{\frac{1}{2}}(3)\left(\int_{\mathfrak{m}_{i}}|S_{3}(\alpha_{i},U_{i})S_{3}^{9}(\alpha_{i},V_{i})|\mathrm{d}\alpha_{i}\right)^{\frac{1}{2}}\\
  &\ll N_{i}^{\frac{11}{36}+\epsilon}\mathscr{J}^{\frac{1}{2}}(3)\left(\int_{\mathfrak{m}_{2}}|S_{3}^{8}(\alpha_{i},V_{i})|\mathrm{d}\alpha_{i}\right)^{\frac{1}{2}}\\
  &\ll N_{i}^{\frac{41}{36}+\epsilon}\mathscr{J}^{\frac{1}{2}}(3).
\end{split}
\end{equation}
Also we apply Lemma \ref{lemfS3} and get
\begin{equation}\label{c}
\begin{split}
  \int_{\mathfrak{m}_{i}}|S_{3}^{4}(\alpha_{i},U_{i})S_{3}^{16}(\alpha_{i},V_{i})|\mathrm{d}\alpha_{i}
  \ll N^{\frac{22}{9}+\epsilon}\mathscr{J}(3).
\end{split}
\end{equation}
Inserting \eqref{c.5}, \eqref{c.6} and \eqref{c} into \eqref{H1H2}, we have
\begin{equation*}
  \mathscr{J}(3)\ll N^{\frac{103}{72}+\epsilon}\mathscr{J}^{\frac{1}{2}}(3),
\end{equation*}
from which we can get this lemma.
\end{proof}

\begin{lemma}\label{lemS26}
We have
\begin{equation*}
  \int_{0}^{1}|S_{3}^{2}(\alpha_{2},U_{i})S_{3}^{6}(\alpha_{i},W_{i})|\mathrm{d}\alpha_{i}\ll N_{i}^{\frac{4}{3}+\epsilon}.
\end{equation*}
\end{lemma}

\begin{proof}
  We can deduce this lemma from (5.6) and (5.9) in Br\"{u}dern \cite{BJ}.
\end{proof}

\begin{lemma}\label{lemlambda}
Let $\operatorname{meas}(\mathscr{E}_{\lambda})$ denote the Lebesgue measure of $\mathscr{E}_{\lambda}$. We have
\begin{equation*}
\operatorname{meas}\left(\mathscr{E}_{\lambda}\right) \ll N_{i}^{-E(\lambda)}
\end{equation*}
with $E(0.87045114)>\frac{13}{18}+10^{-10}$.
\end{lemma}

\begin{proof}
  Taking $\xi=1.16$ and $h=28$ in Lemma 2.3 in Liu and L\"{u}\cite{LL2004}, we can get this lemma.
\end{proof}

\begin{lemma}\label{lemSSS}
We have
\begin{equation}\label{SG}
  \iint\limits_{[0,1]^{2}}|f^{2}(\alpha_{1},N_{1})f^{2}(\alpha_{2},N_{2})G^{4}(\alpha_{1}+\alpha_{2})|\mathrm{d}\alpha_{1}\mathrm{d}\alpha_{2}\leq 305.8869N_{1}N_{2}L^{4},
\end{equation}
\begin{equation}\label{SSW224}
\int_{0}^{1}\left|S_{3}^{2}(\alpha_{i},U_{i})S_{3}^{2}(\alpha_{i},V_{i})S_{3}^{4}(\alpha_{i},W_{i})\right| \mathrm{d} \alpha_{i} \leq 12.677988 U_{i}^{-1}V_{i}^{2}W_{i}^{4}.
\end{equation}
\end{lemma}

\begin{proof}
From the proof of Lemma 2.3 in Kong and Liu \cite[p. 206]{KL}, we can get \eqref{SG}. The estimate \eqref{SSW224} can be found in Lemma 3.6 of the authors \cite{HLDM}.
\end{proof}

\section{Proof of Theorem \ref{thm1} }
\label{sec.proof}

In this section, we complete the proof of Theorem \ref{thm1} by showing desired estimates for $ \mathrm{R}_{r}\left(N_{1}, N_{2}\right), r=1,2,3$. We first estimate $\mathrm{R}_{1}(N_{1},N_{2})$. By Lemma \ref{lem1}, Lemma \ref{lemJn} and Lemma \ref{lemSn}, we have
\begin{equation}\label{R1}
\begin{split}
\mathrm{R}_{1}&(N_{1}, N_{2})\\
=&\iint\limits_{\mathfrak{M}} f(\alpha_{1}, N_{1})S_{3}(\alpha_{1}, U_{1})S_{3}(\alpha_{1}, V_{1})S_{3}^{2}(\alpha_{1},W_{1})
f(\alpha_{2}, N_{2})S_{3}(\alpha_{2}, U_{2})\\
&\times S_{3}(\alpha_{2}, V_{2})S_{3}^{2}(\alpha_{2}, W_{2})G^{k}(\alpha_{1}+\alpha_{2}) e(-\alpha_{1} N_{1}-\alpha_{2} N_{2})
\mathrm{d}\alpha_{1}\mathrm{d}\alpha_{2} \\
\geq& \frac{1}{3^{8}}\sum_{\substack{n_{1}\in\Xi(N_{1},k)\\n_{2}\in\Xi(N_{2},k)\\ n_{1} \equiv n_{2} \equiv 1 \pmod 2}}
\mathfrak{S}(n_{1})\mathfrak{S}(n_{2})\mathfrak{J}(n_{1})\mathfrak{J}(n_{2})\\
\geq& \frac{3.71280584\cdot3^{8}(\sqrt[3]{7}-1)^{2}(1-\eta)^{8}}{3^{8}}U_{1}U_{2}V_{1}V_{2}W_{1}^{2}W_{2}^{2}L^{k}\\
\geq& 3.09441331(1-\eta)^{8}U_{1}U_{2}V_{1}V_{2}W_{1}^{2}W_{2}^{2}L^{k}.
\end{split}
\end{equation}

Next, we turn to handle $\mathrm{R}_{2}(N_{1},N_{2})$. By \eqref{2.2} and \eqref{C(M)}, we obtain
\begin{equation*}
\begin{aligned}
\mathfrak{m} \subset &\left\{\left(\alpha_{1}, \alpha_{2}\right): \alpha_{1} \in \mathfrak{m}_{1}, \alpha_{2} \in[0,1]\right\}
 \cup\left\{\left(\alpha_{1}, \alpha_{2}\right): \alpha_{1} \in[0,1], \alpha_{2} \in \mathfrak{m}_{2}\right\}.
\end{aligned}
\end{equation*}
From the trivial bound $G(\alpha_1+\alpha_{2})\ll L$, we have
\begin{equation}\label{4.2}
\begin{aligned}
\mathrm{R}_{2}&\left(N_{1}, N_{2}\right) \\
=&\iint\limits_{\mathfrak{m} \cap \mathscr{E}_{\lambda}}f(\alpha_{1}, N_{1})S_{3}(\alpha_{1}, U_{1})S_{3}(\alpha_{1}, V_{1})S_{3}^{2}(\alpha_{1},W_{1})
f(\alpha_{2}, N_{2})S_{3}(\alpha_{2}, U_{2})\\
&\times S_{3}(\alpha_{2}, V_{2})S_{3}^{2}(\alpha_{2}, W_{2})G^{k}(\alpha_{1}+\alpha_{2}) e(-\alpha_{1} N_{1}-\alpha_{2} N_{2})
\mathrm{d}\alpha_{1}\mathrm{d}\alpha_{2} \\
\ll& L^{k}\left(\iint\limits_{\substack{\left(\alpha_{1}, \alpha_{2}\right) \in \mathfrak{m}_{1} \times[0,1] \\
\left|G\left(\alpha_{1}+\alpha_{2}\right)\right| \geq \lambda L}}+\iint\limits_{\substack{\left(\alpha_{1}, \alpha_{2}\right) \in[0,1] \times \mathfrak{m}_{2} \\ \left|G\left(\alpha_{1}+\alpha_{2}\right)\right| \geq \lambda L}}\right)|f(\alpha_{1}, N_{1})S_{3}(\alpha_{1}, U_{1})S_{3}(\alpha_{1}, V_{1})S_{3}^{2}(\alpha_{1},W_{1})\\
&\times f(\alpha_{2}, N_{2})S_{3}(\alpha_{2}, U_{2})S_{3}(\alpha_{2}, V_{2})S_{3}^{2}(\alpha_{2}, W_{2})|\mathrm{d}\alpha_{1} \mathrm{d}\alpha_{2} \\
=& L^{k}\left(\int_{0}^{1}\left|f(\alpha_{1}, N_{1})S_{3}(\alpha_{1}, U_{1})S_{3}(\alpha_{1}, V_{1})S_{3}^{2}(\alpha_{1},W_{1})J_{1}(\alpha_{1})\right|\mathrm{d}\alpha_{1}\right.\\
&\left.+\int_{0}^{1}\left|f(\alpha_{2}, N_{2})S_{3}(\alpha_{2}, U_{2})S_{3}(\alpha_{2}, V_{2})S_{3}^{2}(\alpha_{2}, W_{2})J_{2}(\alpha_{2})\right|\mathrm{d}\alpha_{2}\right),
\end{aligned}
\end{equation}
where
\begin{equation*}
  \begin{split}
     J_{1}(\alpha_{1})&=\int\limits_{\substack{\alpha_{2}\in\mathfrak{m}_{2}\\|G(\alpha_{1}+\alpha_{2})|\geq \lambda L}}
     \left|f(\alpha_{2}, N_{2})S_{3}(\alpha_{2}, U_{2})S_{3}(\alpha_{2}, V_{2})S_{3}^{2}(\alpha_{2},W_{2})\right|\mathrm{d}\alpha_{2}, \\
     J_{2}(\alpha_{2})&=\int\limits_{\substack{\alpha_{1}\in\mathfrak{m}_{1}\\|G(\alpha_{1}+\alpha_{2})|\geq \lambda L}}
     \left|f(\alpha_{1}, N_{1})S_{3}(\alpha_{1}, U_{1})S_{3}(\alpha_{1}, V_{1})S_{3}^{2}(\alpha_{1},W_{1})\right|\mathrm{d}\alpha_{1}.
  \end{split}
\end{equation*}
It is easy to get that
\begin{equation*}
\int_{0}^{1}\left|f^{2}(\alpha_{i},N_{i})\right|\mathrm{d} \alpha_{i}\ll N_{i}^{1+\epsilon}.
\end{equation*}
Then we apply Lemma \ref{lemS39}, Lemma \ref{lemS26}, H\"{o}lder's inequality, the periodicity of the function $G(\alpha)$ and get\
\begin{equation}\label{J1}
\begin{split}
  J_{1}(\alpha_{1})\ll&\left(\int_{0}^{1}|f^{2}(\alpha_{2},N_{2})|\mathrm{d}\alpha_{2}\right)^{\frac{1}{2}}
  \left(\int_{\mathfrak{m}_{2}}|S_{3}^{3}(\alpha_{2},U_{2})S_{3}^{9}(\alpha_{2},V_{2})|\mathrm{d}\alpha_{2}\right)^{\frac{1}{9}}\\
  &\times\left(\int_{0}^{1}|S_{3}^{2}(\alpha_{2},U_{2})S_{3}^{6}(\alpha_{2},W_{2})|\mathrm{d}\alpha_{2}\right)^{\frac{1}{3}}
  \left(\int\limits_{\substack{\alpha_{2}\in\mathfrak{m}_{2}\\|G(\alpha_{1}+\alpha_{2})|\geq \lambda L}}1\mathrm{d}\alpha_{2}\right)^{\frac{1}{18}}\\
  \ll& N_{2}^{\frac{409}{324}+\epsilon}\left(\int\limits_{\substack{\omega \in\left[\alpha_{2}, 1+\alpha_{2}\right] \\|G(\omega)| \geq \lambda L}} 1\mathrm{~d} \omega\right)^{\frac{1}{18}},
\end{split}
\end{equation}
where $\omega=\alpha_{1}+\alpha_{2}$.
By \eqref{SSW224}, \eqref{J1}, Lemma \ref{lemlambda} and Cauchy's inequality,  we have
\begin{equation}\label{J11}
\begin{split}
  &\int_{0}^{1}\left|f(\alpha_{1}, N_{1})S_{3}(\alpha_{1}, U_{1})S_{3}(\alpha_{1}, V_{1})S_{3}^{2}(\alpha_{1},W_{1})J_{1}(\alpha_{1})\right|\mathrm{d}\alpha_{1}\\
  \ll& N_{2}^{\frac{409}{324}+\epsilon}(\operatorname{meas}(\mathscr{E}_{\lambda}))^{\frac{1}{18}}\left(\int_{0}^{1}\left|f^{2}(\alpha_{1},N_{1})\right|\mathrm{d} \alpha_{1}\right)^{\frac{1}{2}}\\
  \times&\left(\int_{0}^{1}\left|S_{3}^{2}(\alpha_{1}, U_{1})S_{3}^{2}(\alpha_{1}, V_{1})S_{3}^{4}(\alpha_{1},W_{1})\right|\mathrm{d} \alpha_{1}\right)^{\frac{1}{2}}\\
  \ll& N_{1}^{\frac{11}{9}-10^{-12}}N_{2}^{\frac{11}{9}+\epsilon},
\end{split}
\end{equation}
where $\lambda=0.87045114$ and $ N_{1}\asymp N_{2}$ is used.
Arguing similarly we can also get
\begin{equation}\label{J22}
  \int_{0}^{1}\left|f(\alpha_{2}, N_{2})S_{3}(\alpha_{2}, U_{2})S_{3}(\alpha_{2}, V_{2})S_{3}^{2}(\alpha_{2},W_{2})J_{2}(\alpha_{2})\right|\mathrm{d}\alpha_{2}
  \ll N_{1}^{\frac{11}{9}+\epsilon}N_{2}^{\frac{11}{9}-10^{-12}}.
\end{equation}
Recalling \eqref{UVW1} and inserting  \eqref{J11} and \eqref{J22} into \eqref{4.2},  we have
\begin{equation}\label{R2}
\begin{split}
 \mathrm{R}_{2}(N_{1},N_{2})
 &\ll L^{k} \left( N_{1}^{\frac{11}{9}-10^{-12}}N_{2}^{\frac{11}{9}+\epsilon}
 +N_{1}^{\frac{11}{9}+\epsilon}N_{2}^{\frac{11}{9}-10^{-12}}\right) \\
  &\ll U_{1}U_{2}V_{1}V_{2}W_{1}^{2}W_{2}^{2}L^{k-1}.
\end{split}
\end{equation}

Finally, we handle $\mathrm{R}_{3}(N_{1},N_{2})$. By \eqref{UVW1}, Lemma \ref{lemSSS}, Cauchy's inequality and the definition of $\mathscr{E}_{\lambda}$, we get
\begin{equation}\label{R3}
\begin{aligned}
\mathrm{R}_{3}&\left (N_{1}, N_{2}\right) \\
=&\iint\limits_{\mathfrak{m} \backslash \mathscr{E}_{\lambda}}f(\alpha_{1}, N_{1})S_{3}(\alpha_{1}, U_{1})S_{3}(\alpha_{1}, V_{1})S_{3}^{2}(\alpha_{1},W_{1})
f(\alpha_{2}, N_{2})S_{3}(\alpha_{2}, U_{2})\\
&\times S_{3}(\alpha_{2}, V_{2})S_{3}^{2}(\alpha_{2}, W_{2})G^{k}(\alpha_{1}+\alpha_{2}) e(-\alpha_{1} N_{1}-\alpha_{2} N_{2})
\mathrm{d}\alpha_{1}\mathrm{d}\alpha_{2} \\
\leq&(\lambda L)^{k-2}\left(\iint\limits_{[0,1]^{2}}|f(\alpha_{1},N_{1})f(\alpha_{2},N_{2})G^{2}(\alpha_{1}+\alpha_{2})|^{2}\mathrm{d}\alpha_{1}\mathrm{d}\alpha_{2}\right)^{\frac{1}{2}}\\
&\times\left(\iint\limits_{[0,1]^{2}}|S_{3}(\alpha_{1}, U_{1})S_{3}(\alpha_{1}, V_{1})S_{3}^{2}(\alpha_{1},W_{1})S_{3}(\alpha_{2}, U_{2})S_{3}(\alpha_{2}, V_{2})S_{3}^{2}(\alpha_{2}, W_{2})|^{2}\mathrm{d}\alpha_{1}\mathrm{d}\alpha_{2}\right)^{\frac{1}{2}}\\
\leq&(\lambda L)^{k-2} \left(\iint\limits_{[0,1]^{2}}|f^{2}(\alpha_{1},N_{1})f^{2}(\alpha_{2},N_{2})G^{4}(\alpha_{1}+\alpha_{2})|\mathrm{d}\alpha_{1}\mathrm{d}\alpha_{2}\right)^{\frac{1}{2}}\\
&\times\left(\int_{0}^{1}\left|S_{3}^{2}(\alpha_{1},U_{1})S_{3}^{2}(\alpha_{1},V_{1})S_{3}^{4}(\alpha_{1},W_{1})\right|\mathrm{d}\alpha_{1}\right)^{\frac{1}{2}}\left(\int_{0}^{1}\left|S_{3}^{2}(\alpha_{2},U_{2})S_{3}^{2}(\alpha_{2},V_{2})S_{3}^{4}(\alpha_{2},W_{2})\right|\mathrm{d} \alpha_{2}\right)^{\frac{1}{2}}\\
\leq& 221.73322593 \lambda^{k-2}N_{1}^{\frac{1}{2}}N_{2}^{\frac{1}{2}}U_{1}^{-\frac{1}{2}}U_{2}^{-\frac{1}{2}}V_{1}V_{2}W_{1}^{2}W_{2}^{2} L^{k}\\
\leq& 1773.86580744 (1-\eta)^{-1}\lambda^{k-2}U_{1}U_{2}V_{1}V_{2}W_{1}^{2}W_{2}^{2} L^{k}.
\end{aligned}
\end{equation}

Inserting \eqref{R1}, \eqref{R2} and \eqref{R3} into \eqref{R} we have
\begin{equation*}
\begin{aligned}
\mathrm{R}(N_{1}, N_{2})&\geq \mathrm{R}_{1}(N_{1}, N_{2})-\mathrm{R}_{3}(N_{1}, N_{2})+O\left(U_{1}U_{2}V_{1}V_{2}W_{1}^{2}W_{2}^{2}L^{k-1}\right) \\
&>\left(3.09441331(1-\eta)^{8}-1773.86580744 (1-\eta)^{-1}\lambda^{k-2}\right)U_{1}U_{2}V_{1}V_{2}W_{1}^{2}W_{2}^{2}L^{k},
\end{aligned}
\end{equation*}
where $\lambda=0.87045114$.
Then we can deduce $k\geq 48$ from
 $$\mathrm{R}\left(N_{1}, N_{2}\right)>0.$$
Thus we complete the proof of Theorem \ref{thm1}.

\section*{Acknowledgment}
The authors would like to thank the referees for their many useful comments.
This work is supported by National Natural Science Foundation of China (Grant No. 12171286).

\end{document}